\documentclass[pdflatex,sn-mathphys-num]{sn-jnl}

\usepackage{graphicx}%
\usepackage{multirow}%
\usepackage{amsmath,amssymb,amsfonts}%
\usepackage{mathtools}
\usepackage[title]{appendix}%
\usepackage{xcolor}%
\usepackage{textcomp}%
\usepackage{manyfoot}%
\usepackage{booktabs}%
\usepackage{algorithm}%
\usepackage{algorithmicx}%
\usepackage{algpseudocode}%
\usepackage{listings}%
\usepackage{comment}
\usepackage{tikz}
\usepackage{pgfplots}
\pgfplotsset{compat=1.18}
\usepackage{bm}
\usepackage{cleveref}

\theoremstyle{thmstyleone}%
\newtheorem{theorem}{Theorem}%
\newtheorem{proposition}[theorem]{Proposition}%
\newtheorem{lemma}[theorem]{Lemma}

\theoremstyle{thmstyletwo}%
\newtheorem{remark}{Remark}%

\theoremstyle{thmstylethree}%

\newcommand{\ignore}[1]{}

\begin{document}

\title[]{Forest Characterization of When SDP-RLT Strengthens RLT for Bilinear Terms
}

\author[1]{\fnm{Hayato} \sur{Watanabe}}\email{watanabe.h.560f@m.isct.ac.jp}

\author[2]{\fnm{Takami} \sur{Sato}}\email{takami.sato@yos.elec.keio.ac.jp}

\author[1]{\fnm{Kazuhide} \sur{Nakata}}\email{nakata.k.ac@m.titech.ac.jp}

\author*[1]{\fnm{Ken} \sur{Kobayashi}}\email{kobayashi.k@iee.eng.isct.ac.jp}

\affil*[1]{\orgdiv{Department of Industrial Engineering and Economics}, \orgname{Institute of Science Tokyo}, \orgaddress{\street{2-12-1 Ookayama}, \city{Meguro-ku}, \postcode{152-8550}, \state{Tokyo}, \country{Japan}}}

\affil[2]{\orgdiv{Department of Electronics and Electrical Engineering}, \orgname{Keio University}, \orgaddress{\street{3-14-1 Hiyoshi, Kohoku-ku}, \city{Yokohama}, \postcode{223-8522}, \state{Kanagawa}, \country{Japan}}}

\abstract{
We study two relaxations for bilinear terms in box-constrained variables: the reformulation-linearization technique (RLT) relaxation and the SDP-RLT relaxation, which combines RLT with semidefinite programming (SDP).
We represent the pattern of bilinear terms by an undirected graph, with one vertex for each variable and one edge for each bilinear term.
A point feasible for the RLT relaxation fixes only part of the lifted matrix of the SDP-RLT relaxation, leaving the remaining entries to be chosen subject to the diagonal RLT inequalities, and the two relaxations differ exactly when these entries cannot always be chosen so that the matrix becomes positive semidefinite.
Formulating this as a positive semidefinite matrix completion problem, we show that the feasible regions of the two relaxations coincide if and only if the graph is a forest.
We also measure how large the difference can be when the graph contains a complete bipartite subgraph: a bilinear objective given by a Hadamard matrix yields optimal values whose ratio grows with the size of that subgraph.
}


\maketitle

\section{Introduction}\label{sec1}

We consider optimization problems involving bilinear terms $x_ix_j$ between distinct components of box-constrained variables $\bm{x}\in\mathbb{R}^{n}$.
A common reformulation introduces an auxiliary variable $w_{ij}$ for each bilinear term, so that the nonconvexity is concentrated in the equalities $w_{ij}=x_ix_j$, which are then replaced by convex relaxations.
The reformulation-linearization technique (RLT)~\citep{sherali1995reformulation,tawarmalani2002convexification} gives a linear relaxation consisting of the McCormick inequalities~\citep{mccormick1976computability} for each bilinear term.
The semidefinite programming (SDP) relaxation instead lifts $(1,\bm{x}^{\top})^{\top}$ to a symmetric matrix variable and imposes a positive semidefinite constraint while dropping the rank-one condition~\citep{bao2011sdp}.
The SDP-RLT relaxation imposes the McCormick inequalities, the positive semidefinite constraint, and the diagonal RLT inequalities together~\citep{anstreicher2009sdp,bao2011sdp}, and hence its feasible region is contained in that of the RLT relaxation.

In this paper, we examine when the SDP-RLT relaxation is strictly stronger than the RLT relaxation from the perspective of positive semidefinite completion.
A point feasible for the RLT relaxation is also feasible for the SDP-RLT relaxation if and only if the lifted matrix can be made positive semidefinite.
For a point feasible for the RLT relaxation, the entries of the lifted matrix corresponding to $\bm{x}$ and the bilinear terms of the problem are determined, while the remaining entries may be chosen subject to the diagonal RLT inequalities.
The two relaxations therefore differ exactly when, at some RLT-feasible point, no such choice of the remaining entries makes the matrix positive semidefinite, which is a positive semidefinite completion problem.

To state our main result, we represent the pattern of bilinear terms by a simple undirected graph $G(V,E)$, whose vertices $V=\{1,\dots,n\}$ index the variables and whose edges correspond to the bilinear terms appearing in the problem.
Our main result is an exact graph characterization: the feasible regions of the RLT and SDP-RLT relaxations coincide if and only if $G(V,E)$ is a forest.
Known results on the Boolean quadric polytope imply that, for a forest, the RLT relaxation describes the convex hull of the bilinear constraint set, and hence the two relaxations coincide~\citep{padberg1989boolean}.
However, for a graph containing a cycle, the failure of RLT to describe this convex hull does not by itself imply that SDP-RLT is strictly stronger.
In this paper, we establish the converse by constructing, for every graph containing a cycle, an RLT-feasible point whose associated partial matrix has no positive semidefinite completion.
For the forest case, we give a direct proof using the classical completion theorem for chordal patterns~\citep{grone1984completion}, thereby treating both directions within the same completion framework.
We also measure how much the two feasible regions can differ when $G(V,E)$ is not a forest.
We show that if $G(V,E)$ contains a complete bipartite subgraph, then for a certain bilinear objective function, the ratio between the optimal values of the two relaxations grows with the size of that subgraph.

The rest of this paper is organized as follows.
\Cref{sec:completion} collects the notions and the known results on positive semidefinite matrix completion.
\Cref{sec:relaxations} introduces the RLT and SDP-RLT relaxations for bilinear terms in box-constrained variables.
\Cref{sec:characterization} proves the forest--cycle characterization, and \Cref{sec:hadamard-illustration} measures the gap between the two relaxations when $G$ contains a complete bipartite subgraph.
\Cref{sec:conclusion} concludes the paper.

\subsection{Related work}

The convex envelopes of multilinear functions and the strength of McCormick relaxations have been studied in several works~\citep{rikun1997convex,luedtke2012some,boland2017gap}.
In particular, \citet{boland2017gap} showed that the McCormick relaxation describes the convex hull of the graph of a bilinear function if and only if every cycle has an even number of positive and an even number of negative edges, and used a Hadamard matrix to construct instances on which the two differ.
Related questions for binary variables have been studied through the Boolean quadric polytope~\citep{padberg1989boolean,Gupte2020,delpia2018multilinear}.
These works compare the McCormick relaxation with the convex hull, whereas we compare the RLT and SDP-RLT relaxations.

For nonconvex quadratically constrained quadratic programs, \citet{anstreicher2009sdp} compared the RLT and SDP-RLT relaxations in the lifted space, and showed that, for the five-dimensional lifted region associated with two original variables, adding the positive semidefinite constraint reduces the volume of the RLT feasible region by three quarters.
Several works also give graph-based conditions under which SDP relaxations are exact~\citep{sojoudi2014exactness,azuma2020forest,azuma2022exact}.
More recently, \citet{qiu2024exact} gave algebraic characterizations of box-constrained quadratic programs whose RLT or SDP-RLT relaxation is exact.
Their notion of exactness compares a relaxation value with the optimal value of a fixed nonconvex instance and hence depends on the objective data, whereas we fix the graph of bilinear terms and ask whether the feasible regions of the two relaxations~coincide.

Bipartite bilinear programs have also received considerable attention.
\citet{dey2019socp} derived a second order cone representable relaxation from the convex hull of a set defined by a single bipartite bilinear equality, whose full formulation can require an exponential number of conic constraints; \citet{gu2021lifting} derived lifted bilinear cover inequalities, and \citet{dey2026aggregation} studied relaxations obtained by aggregating bipartite bilinear equality constraints. 
For problems whose bilinear terms are separable, that is, in which each variable appears in at most one bilinear term, \citet{gu2024sparse} proved that the SDP relaxation provides no benefit over the McCormick relaxation in terms of the~optimal~value.

Positive semidefinite matrix completion provides the framework of this paper.
Classical results relate the existence of positive definite completions to whether the graph of specified entries is chordal~\citep{grone1984completion}, and completion theory has since been used to exploit sparsity in semidefinite and polynomial optimization~\citep{fukuda2001exploiting,waki2006sums}.
\citet {naevdal1998cone} studied quantitative comparisons between the cone defined by positive semidefiniteness on the specified cliques and the cone of partial matrices with positive semidefinite completions from an eigenvalue perspective.

\subsection{Notation}

Let $\mathbb{S}^p$ denote the set of $p\times p$ real symmetric matrices.
For a matrix $\bm{A}$, we denote its $(i,j)$-entry by $A_{ij}$, and for $C\subseteq\{1,\ldots,p\}$ and $\bm{A}\in\mathbb{S}^p$, we write $\bm{A}_{CC}$ for the principal submatrix of $\bm{A}$ indexed by $C$.
For $\bm{A}\in\mathbb{S}^p$, the notation $\bm{A}\succeq\bm{O}$ means that $\bm{A}$ is positive semidefinite, and $\operatorname{diag}(\bm{A})$ denotes the vector of its diagonal entries.
For $\bm{A},\bm{B}\in\mathbb{S}^p$, we write $\bm{A}\bullet\bm{B}\coloneqq\sum_{i}\sum_{j}A_{ij}B_{ij}$, and for $\bm{a}\in\mathbb{R}^p$, we write $\|\bm{a}\|$ for the Euclidean norm.
We write $\bm{O}$, $\bm{I}$, $\bm{0}$, and $\bm{1}$ for the zero matrix, the identity matrix, the zero vector, and the all-ones vector, whose sizes are clear from the context.
For positive integers $n_1$ and $n_2$, we denote by $K_{n_1,n_2}$ the complete bipartite graph with two disjoint vertex sets of sizes $n_1$ and $n_2$, and by $K_n$ the complete graph on $n$ vertices.

\section{Positive semidefinite matrix completion}
\label{sec:completion}

This section summarizes the notions and the known result on positive semidefinite matrix completion used in this paper.
Let $G(V,E)$ be an undirected graph with no loops on the vertex set $V\coloneqq\{1,\ldots,n\}$, where we identify $(i,j)\in E$ with $(j,i)\in E$, and let
\[
    E^{\bullet}\coloneqq E\cup\{(i,i):i\in V\}.
\]
A symmetric matrix of order $n$ is called a \emph{partial symmetric matrix} if the entries in $E^{\bullet}$ are specified and the others are unspecified, and the set of such matrices is denoted by $\mathbb{S}^n(E^{\bullet},?)$, following \citet{fukuda2001exploiting}.
A \emph{completion} of $\bm{X}\in\mathbb{S}^n(E^{\bullet},?)$ is a symmetric matrix $\bar{\bm{X}}\in\mathbb{S}^n$ with $\bar{X}_{ij}=X_{ij}$ for $(i,j)\in E^{\bullet}$, and it is called a \emph{positive semidefinite completion} if $\bar{\bm{X}}\succeq\bm{O}$.
We write
\[
    \mathbb{S}^n_+(E^{\bullet},?)
    \coloneqq
    \bigl\{
        \bm{X}\in\mathbb{S}^n(E^{\bullet},?)
        \colon
        \bm{X} \text{ has a positive semidefinite completion}
    \bigr\}.
\]

We recall two notions from graph theory.
A graph is \emph{complete} if every pair of distinct vertices is connected by an edge, and a \emph{clique} of $G(V,E)$ is an induced subgraph that is complete.
A clique is \emph{maximal} if no other clique contains it.
In what follows, we call a vertex subset $C\subseteq V$ a clique of $G(V,E)$ whenever it induces a clique.
The graph $G(V,E)$ is \emph{chordal} if every cycle of length at least four has a chord, that is, an edge joining two vertices that are not consecutive on the cycle.
For every clique $C$ of $G(V,E)$ and every $\bm{X}\in\mathbb{S}^n(E^{\bullet},?)$, all entries of $\bm{X}_{CC}$ are specified, so $\bm{X}_{CC}\succeq\bm{O}$ holds for every clique $C$ whenever $\bm{X}\in\mathbb{S}^n_+(E^{\bullet},?)$.
The following classical result states that, for chordal graphs, the converse holds as well, and that it suffices to consider the maximal cliques.

\begin{theorem}[\citet{grone1984completion}]
    \label{thm:grone}
    Let $C_1,\ldots,C_{\ell}$ be the maximal cliques of a chordal graph $G(V,E)$.
    Then $\bm{X}\in\mathbb{S}^n(E^{\bullet},?)$ belongs to $\mathbb{S}^n_+(E^{\bullet},?)$ if and only if $\bm{X}_{C_rC_r}\succeq\bm{O}$ for every $r=1,\ldots,\ell$.
\end{theorem}

\section{Two relaxations for bilinear terms}
\label{sec:relaxations}

Let $\bm{x}\in\mathbb{R}^{n}$ be a vector of box-constrained variables indexed by $V$.
Let $E$ be the set of index pairs $(i,j)$ with $i\neq j$ for which the bilinear term $x_ix_j$ appears in the problem, so that each edge of $G(V,E)$ corresponds to such a term.
For each $(i,j)\in E$, let $w_{ij}$ be a lifted variable representing the product $x_ix_j$, and let $\bm{w}=(w_{ij})_{(i,j)\in E}\in\mathbb{R}^{|E|}$ collect these variables.
This paper is concerned with convex relaxations of the nonconvex set
\begin{equation}
    \mathcal{F}
    \coloneqq
    \left\{
        (\bm{x},\bm{w})
        \colon
        w_{ij}=x_ix_j~((i,j)\in E),~\bm{x}\in[-1,1]^{n}
    \right\}.
    \label{eq:bilinear-set}
\end{equation}
Throughout this paper, we assume without loss of generality that each component of $\bm{x}$ ranges over $[-1,1]$.
The only nonconvex constraints in \eqref{eq:bilinear-set} are the equalities $w_{ij}=x_ix_j$, and the two relaxations below differ in how these constraints are replaced.

For each $(i,j)\in E$, multiplying the bounds on $x_i$ and on $x_j$ in all four
sign combinations and replacing $x_ix_j$ by $w_{ij}$ gives the McCormick
inequalities~\citep{mccormick1976computability} as follows:
\begin{equation}
    \left\{
    \begin{aligned}
        w_{ij}&\geq x_i+x_j-1,\\
        w_{ij}&\geq -x_i-x_j-1,\\
        w_{ij}&\leq -x_i+x_j+1,\\
        w_{ij}&\leq x_i-x_j+1
    \end{aligned}
    \right.
    \qquad ((i,j)\in E).
    \label{eq:mc}
\end{equation}
The RLT relaxation of $\mathcal{F}$ is then formulated as
\begin{equation}
    \mathcal{F}_{\mathrm{RLT}}
    \coloneqq
    \left\{
        (\bm{x},\bm{w})
        \colon
            \bm{x}\in[-1,1]^{n}, \text{\Cref{eq:mc}}
    \right\}.
    \label{eq:rlt}
\end{equation}

For the second relaxation, we introduce a matrix variable
$\bm{W}\in\mathbb{S}^{n}$ representing $\bm{x}\bm{x}^{\top}$
and link it to the lifted variables by
\begin{equation}
    W_{ij}=w_{ij}\qquad ((i,j)\in E).
    \label{eq:w-link}
\end{equation}
With $\bm{x}$ fixed, we define the associated lifted matrix by
\begin{equation}
    \bm{M}(\bm{W})
    \coloneqq
    \begin{pmatrix}
        1 & \bm{x}^{\top}\\
        \bm{x} & \bm{W}
    \end{pmatrix}.
    \label{eq:lifted-matrix}
\end{equation}
Suppose that $(\bm{x},\bm{w})\in\mathcal{F}$ and $\bm{W}=\bm{x}\bm{x}^{\top}$. 
Then $\bm{M}(\bm{W})=(1,\bm{x}^{\top})^{\top}(1,\bm{x}^{\top})$ and is therefore positive semidefinite of rank one, and $W_{ii}=x_i^2\leq1$ for every $i\in V$.
Dropping the nonconvex rank-one condition leaves the convex constraint
\begin{equation}
    \bm{M}(\bm{W})\succeq\bm{O},
    \label{eq:psd}
\end{equation}
and the second property gives the diagonal RLT inequalities
\begin{equation}
    \operatorname{diag}(\bm{W})\leq\bm{1}.
    \label{eq:diag-rlt}
\end{equation}
The SDP-RLT relaxation is then defined by
\begin{equation}
    \mathcal{F}_{\mathrm{SDP\text{-}RLT}}
    \coloneqq
    \mathcal{F}_{\mathrm{RLT}}\cap
    \left\{
        (\bm{x},\bm{w})
        \colon
        \begin{aligned}
            &\exists\,\bm{W}\text{ satisfying}\\
            &\text{\Cref{eq:w-link,eq:psd,eq:diag-rlt}}
        \end{aligned}
    \right\}.
    \label{eq:sdp-rlt}
\end{equation}

\section{Forest--cycle characterization}
\label{sec:characterization}

This section proves the main result of this paper: $\mathcal{F}_{\mathrm{SDP\text{-}RLT}}=\mathcal{F}_{\mathrm{RLT}}$ holds if and only if $G(V,E)$ is a forest (\Cref{thm:forest-cycle}).
We first show that a point $(\bm{x},\bm{w})\in\mathcal{F}_{\mathrm{RLT}}$ belongs to $\mathcal{F}_{\mathrm{SDP\text{-}RLT}}$ if and only if a partial symmetric matrix determined by $(\bm{x},\bm{w})$ has a positive semidefinite completion.
We then show that such a completion exists at every point of $\mathcal{F}_{\mathrm{RLT}}$ if and only if $G(V,E)$ is a forest.

Taking the Schur complement of $\bm{M}(\bm{W})$, we have
\begin{equation}
    \bm{M}(\bm{W})\succeq\bm{O}
    \quad\Longleftrightarrow\quad
    \bm{W}-\bm{x}\bm{x}^{\top}\succeq\bm{O},
    \label{eq:schur-complement}
\end{equation}
and we write
$\bm{S}\coloneqq\bm{W}-\bm{x}\bm{x}^{\top}\in\mathbb{S}^n$.
From \Cref{eq:sdp-rlt,eq:schur-complement}, a point
$(\bm{x},\bm{w})\in\mathcal{F}_{\mathrm{RLT}}$
belongs to $\mathcal{F}_{\mathrm{SDP\text{-}RLT}}$
if and only if there is an $\bm{S}\succeq\bm{O}$ for which
\begin{equation}
    \bm{W}=\bm{S}+\bm{x}\bm{x}^{\top}
    \label{eq:w-from-s}
\end{equation}
satisfies \Cref{eq:w-link,eq:diag-rlt}.
We now examine how \Cref{eq:w-link,eq:diag-rlt}
constrain the entries of $\bm{S}$.
For a fixed $(\bm{x},\bm{w})\in\mathcal{F}_{\mathrm{RLT}}$,
\Cref{eq:w-link} determines the entries of $\bm{S}$
indexed by $E$ as follows:
\begin{equation}
    S_{ij}=w_{ij}-x_i x_j
    \qquad ((i,j)\in E),
    \label{eq:schur-entries}
\end{equation}
and \Cref{eq:diag-rlt} gives the upper bounds
$S_{ii}\le 1-x_i^2$ for $i\in V$ on the diagonal entries.
The other entries are free.
The diagonal entries are bounded above rather than fixed,
but they may be set to their upper bounds without affecting
the existence of $\bm{S}\succeq\bm{O}$,
as shown in the proof of the next lemma.
We therefore let
$\bm{S}^*\in\mathbb{S}^n(E^\bullet,?)$
be the partial symmetric matrix determined by
$(\bm{x},\bm{w})$, with specified entries
\begin{equation}
    \begin{aligned}
        S^*_{ii}&\coloneqq 1-x_i^2
            && (i\in V),\\
        S^*_{ij}&\coloneqq w_{ij}-x_i x_j
            && ((i,j)\in E).
    \end{aligned}
    \label{eq:partial-matrix}
\end{equation}
Thus, the diagonal entries of $\bm{S}^*$ are fixed;
only the off-diagonal entries outside $E$ are unspecified.
The following lemma characterizes membership in
$\mathcal{F}_{\mathrm{SDP\text{-}RLT}}$
as a completion problem for $\bm{S}^*$.

\begin{lemma}
    \label{lem:sdp-rlt-completion}
    For every
    $(\bm{x},\bm{w})\in\mathcal{F}_{\mathrm{RLT}}$,
    the condition
    $(\bm{x},\bm{w})\in\mathcal{F}_{\mathrm{SDP\text{-}RLT}}$
    is equivalent to
    $\bm{S}^*\in\mathbb{S}_+^n(E^\bullet,?)$.
\end{lemma}

\begin{proof}
    Suppose that
    $(\bm{x},\bm{w})\in\mathcal{F}_{\mathrm{SDP\text{-}RLT}}$.
    Then there is an $\bm{S}\succeq\bm{O}$
    whose entries in $E$ are given by \Cref{eq:schur-entries}
    and whose diagonal entries satisfy
    $S_{ii}\le 1-x_i^2$ for $i\in V$.
    Let $\bm{D}$ be the diagonal matrix with entries
    $1-x_i^2-S_{ii}$ for $i\in V$.
    Then $\bm{D}\succeq\bm{O}$,
    so $\bm{S}+\bm{D}\succeq\bm{O}$,
    and adding $\bm{D}$ changes no off-diagonal entry.
    Moreover,
    $S_{ii}+D_{ii}=1-x_i^2=S^*_{ii}$
    for $i\in V$.
    Hence $\bm{S}+\bm{D}$ is a positive semidefinite
    completion of $\bm{S}^*$, and
    $\bm{S}^*\in\mathbb{S}_+^n(E^\bullet,?)$.

    Conversely, let $\bar{\bm{S}}\succeq\bm{O}$
    be a completion of $\bm{S}^*$.
    In particular,
    $\bar{S}_{ii}=S^*_{ii}=1-x_i^2$ for $i\in V$.
    Then $\bm{W}=\bar{\bm{S}}+\bm{x}\bm{x}^{\top}$
    satisfies \Cref{eq:w-link,eq:psd,eq:diag-rlt},
    and  we have $(\bm{x},\bm{w})\in\mathcal{F}_{\mathrm{SDP\text{-}RLT}}$.
\end{proof}

For each $(i,j)\in E$, the vertex set $\{i,j\}$
is a clique of $G(V,E)$, so the $2\times2$ principal
submatrix of $\bm{S}^*$ indexed by $\{i,j\}$
must be positive semidefinite for
$\bm{S}^*\in\mathbb{S}_+^n(E^\bullet,?)$ to hold.
The next lemma shows that this condition is always satisfied.

\begin{lemma}
    \label{lem:edge-psd}
    For every $(\bm{x},\bm{w})\in\mathcal{F}_{\mathrm{RLT}}$
    and each $(i,j)\in E$, the $2\times2$ principal submatrix
    of $\bm{S}^*$ indexed by $\{i,j\}$ satisfies
    \begin{equation}
        \bm{S}^*_{\{i,j\}\{i,j\}}
        =
        \begin{pmatrix}
            1-x_i^2 & w_{ij}-x_i x_j\\
            w_{ij}-x_i x_j & 1-x_j^2
        \end{pmatrix}
        \succeq\bm{O}.
        \label{eq:edge-psd}
    \end{equation}
\end{lemma}

\begin{proof}
    By \Cref{eq:partial-matrix} and $x_i,x_j\in[-1,1]$,
    the diagonal entries $1-x_i^2$ and $1-x_j^2$
    are nonnegative.
    Thus, it suffices to show that the determinant
    is nonnegative, that is,
    \[
        |w_{ij}-x_i x_j|
        \le \sqrt{(1-x_i^2)(1-x_j^2)}.
    \]
    By \Cref{eq:mc}, the upper bounds on $w_{ij}$ give
    \[
        w_{ij}-x_i x_j
        \le \min\{(1-x_i)(1+x_j),(1+x_i)(1-x_j)\}
        \le \sqrt{(1-x_i^2)(1-x_j^2)},
    \]
    and the lower bounds give
    \[
        x_i x_j-w_{ij}
        \le \min\{(1+x_i)(1+x_j),(1-x_i)(1-x_j)\}
        \le \sqrt{(1-x_i^2)(1-x_j^2)},
    \]
    where each second inequality follows from
    $\min\{a,b\}\le\sqrt{ab}$ for $a,b\ge0$.
\end{proof}

We now state the main result.

\begin{theorem}[Forest--cycle characterization]
    \label{thm:forest-cycle}
    $\mathcal{F}_{\mathrm{SDP\text{-}RLT}}
    =\mathcal{F}_{\mathrm{RLT}}$
    holds if and only if $G(V,E)$ is a forest.
\end{theorem}

\begin{proof}
    First, suppose that $G(V,E)$ is a forest.
    We show
    $\mathcal{F}_{\mathrm{SDP\text{-}RLT}}
    =\mathcal{F}_{\mathrm{RLT}}$.
    By \Cref{eq:sdp-rlt} and Lemma~\ref{lem:sdp-rlt-completion},
    it suffices to show
    $\bm{S}^*\in\mathbb{S}_+^n(E^\bullet,?)$
    for every
    $(\bm{x},\bm{w})\in\mathcal{F}_{\mathrm{RLT}}$.
    Thus, let
    $(\bm{x},\bm{w})\in\mathcal{F}_{\mathrm{RLT}}$
    be fixed.
    A forest is chordal, and its maximal cliques are the pairs of adjacent vertices and the isolated vertices.
    For an edge $(i,j)\in E$, the corresponding submatrix
    of $\bm{S}^*$ indexed by $i$ and $j$ is positive semidefinite
    by Lemma~\ref{lem:edge-psd}.
    For an isolated vertex $i$, the corresponding submatrix
    is the single diagonal entry $S^*_{ii}=1-x_i^2\ge0$.
    Hence \Cref{thm:grone} gives
    $\bm{S}^*\in\mathbb{S}_+^n(E^\bullet,?)$.

    Next, we show the converse by contraposition:
    if $G(V,E)$ is not a forest, then
    $\mathcal{F}_{\mathrm{SDP\text{-}RLT}}
    \subsetneq\mathcal{F}_{\mathrm{RLT}}$.
    Suppose that $G(V,E)$ is not a forest,
    and let $i_1,i_2,\ldots,i_k$ with $k\ge3$
    be distinct vertices such that
    $(i_\ell,i_{\ell+1})\in E$ for $\ell=1,\ldots,k-1$
    and $(i_k,i_1)\in E$, that is, they form a cycle
    of $G(V,E)$.
    We construct a point of $\mathcal{F}_{\mathrm{RLT}}$
    that does not belong to
    $\mathcal{F}_{\mathrm{SDP\text{-}RLT}}$.
    Let $\bm{x}=\bm{0}$, and
    \begin{equation}
        \label{eq:cycle-point}
        w_{i_\ell i_{\ell+1}}=1
        \quad (\ell=1,\ldots,k-1),
        \qquad
        w_{i_k i_1}=-1,
    \end{equation}
    and $w_{ij}=0$ for every edge $(i,j)\in E$
    that does not belong to the cycle.
    A similar construction was used by
    \citet[proof of Theorem~7]{grone1984completion}.
    At $\bm{x}=\bm{0}$ the McCormick inequalities
    \eqref{eq:mc} reduce to $-1\le w_{ij}\le1$,
    and hence
    $(\bm{x},\bm{w})\in\mathcal{F}_{\mathrm{RLT}}$.
    For this point, \Cref{eq:partial-matrix} gives
    $S^*_{ii}=1$ for $i\in V$ and
    $S^*_{ij}=w_{ij}$ for $(i,j)\in E$.

    For the point constructed above, we show
    $\bm{S}^*\notin\mathbb{S}_+^n(E^\bullet,?)$
    by contradiction.
    Suppose that
    $\bm{S}^*\in\mathbb{S}_+^n(E^\bullet,?)$,
    and let $\bar{\bm{S}}$ be a positive semidefinite
    completion of $\bm{S}^*$.
    Then there is a matrix
    $\bm{U}\in\mathbb{R}^{n\times n}$ satisfying
    $\bar{\bm{S}}=\bm{U}^{\top}\bm{U}$,
    and we write the columns of $\bm{U}$ as
    $\bm{u}_i$ for $i\in V$.
    Since $\bar{\bm{S}}$ is a completion of $\bm{S}^*$,
    whose diagonal entries are all one, we have
    $\|\bm{u}_i\|=1$ for $i\in V$, and
    $\bm{u}_i^{\top}\bm{u}_j=w_{ij}$ for $(i,j)\in E$.
    Focusing on the edges of the cycle,
    \Cref{eq:cycle-point} gives
    \[
        \bm{u}_{i_1}=\bm{u}_{i_2}=\cdots=\bm{u}_{i_k}
        \qquad\text{and}\qquad
        \bm{u}_{i_k}=-\bm{u}_{i_1}.
    \]
    Therefore, we have
    $\bm{u}_{i_1}=-\bm{u}_{i_1}$,
    that is, $\bm{u}_{i_1}=\bm{0}$,
    which contradicts $\|\bm{u}_{i_1}\|=1$.
    Hence $\bm{S}^*\notin\mathbb{S}_+^n(E^\bullet,?)$,
    and Lemma~\ref{lem:sdp-rlt-completion} gives
    $(\bm{x},\bm{w})\notin
    \mathcal{F}_{\mathrm{SDP\text{-}RLT}}$,
    which completes the proof.
\end{proof}

\begin{remark}
    \label{rem:polyhedral}
    The equality $\mathcal{F}_{\mathrm{SDP\text{-}RLT}}=\mathcal{F}_{\mathrm{RLT}}$ for a forest also follows from known results on the Boolean quadric polytope, giving an alternative proof of one direction of \Cref{thm:forest-cycle}.
    \citet{padberg1989boolean} showed that this polytope is described by the McCormick inequalities if and only if the graph is acyclic; see also \citet[Section~2.1]{Gupte2020}.
    Since a bilinear function is affine in each variable separately, this gives that $\mathcal{F}_{\mathrm{RLT}}$ coincides with the convex hull of $\mathcal{F}$ when $G(V,E)$ is a forest, and hence no convex relaxation of $\mathcal{F}$ can be strictly stronger than the RLT relaxation.
\end{remark}

\section{Quantifying the gap between the two relaxations}
\label{sec:hadamard-illustration}

\Cref{thm:forest-cycle} characterizes the graphs for which the positive semidefinite constraint strictly strengthens the RLT relaxation, but says nothing about how much the optimal value changes.
This section gives an instance for which the ratio between the two optimal values is large.

An undirected graph $G(V,E)$ is given, and we consider optimization problems whose bilinear terms are those indexed by $E$.
For $\bm{c}=(c_{ij})_{(i,j)\in E}\in\mathbb{R}^{|E|}$, consider the box-constrained bilinear optimization problem
\begin{equation}
    \max\left\{\sum_{(i,j)\in E}c_{ij}x_ix_j \colon \bm{x}\in[-1,1]^{n}\right\}.
    \label{eq:bilinear-max}
\end{equation}
Replacing each product $x_ix_j$ by $w_{ij}$ and relaxing the resulting problem, we define the optimal values
\begin{align}
    v_{\mathrm{R}}(\bm{c})
    &\coloneqq
    \max\left\{\sum_{(i,j)\in E}c_{ij}w_{ij} \colon (\bm{x},\bm{w})\in\mathcal{F}_{\mathrm{RLT}}\right\},
    \label{eq:value-r}\\
    v_{\mathrm{SR}}(\bm{c})
    &\coloneqq
    \max\left\{\sum_{(i,j)\in E}c_{ij}w_{ij} \colon (\bm{x},\bm{w})\in\mathcal{F}_{\mathrm{SDP\text{-}RLT}}\right\}.
    \label{eq:value-rs}
\end{align}

\begin{proposition}
    \label{prop:hadamard-gap}
    Suppose that $G(V,E)$ contains a complete bipartite subgraph with disjoint vertex sets $I_r,J_r\subseteq V$ satisfying $|I_r|=|J_r|=r\geq2$.
    Then there is a $\bm{c}\in\mathbb{R}^{|E|}$ with
    \[
        \frac{v_{\mathrm{R}}(\bm{c})}{v_{\mathrm{SR}}(\bm{c})}>\sqrt{r/2}.
    \]
\end{proposition}

\begin{proof}
   Let $m\coloneqq2^{\lfloor\log_2 r\rfloor}$, so that $r/2<m\leq r$, and let $\bm{H}\in\{-1,1\}^{m\times m}$ denote a Hadamard matrix of order $m$, that is, $\bm{H}\bm{H}^{\top}=m\bm{I}$, which exists by the Sylvester construction. 
   Take subsets $I_m\subseteq I_r$ and $J_m\subseteq J_r$ with $|I_m|=|J_m|=m$, and index the rows of $\bm{H}$ by $I_m$ and its columns by $J_m$, so that $H_{ij}$ is defined for $i\in I_m$ and $j\in J_m$.
   Let $\bm{c}$ be given by
    \begin{equation}
        c_{ij}=H_{ij}/m \quad (i\in I_m,\ j\in J_m),
        \qquad
        c_{ij}=0 \quad (\text{otherwise}).
        \label{eq:hadamard-objective}
    \end{equation}
    For this $\bm{c}$, we show that $v_{\mathrm{R}}(\bm{c})=m$ and $v_{\mathrm{SR}}(\bm{c})=\sqrt m$, which gives $v_{\mathrm{R}}(\bm{c})/v_{\mathrm{SR}}(\bm{c})=\sqrt m>\sqrt{r/2}$.
    We first prove $v_{\mathrm{R}}(\bm{c})=m$.
        Every $(\bm{x},\bm{w})\in\mathcal{F}_{\mathrm{RLT}}$ satisfies $|w_{ij}|\leq1$ by \Cref{eq:mc}, and $|c_{ij}|=1/m$ for $i\in I_m$ and $j\in J_m$, so the objective is at most $m$ and we have  $v_{\mathrm{R}}(\bm{c})\leq m$.
    To see that this bound is attained, let $\bar{\bm{x}}=\bm{0}$ and
    \[
        \bar{w}_{ij}=H_{ij} \quad (i\in I_m,\ j\in J_m),
        \qquad
        \bar{w}_{ij}=0 \quad (\text{otherwise}).
    \]
    At $\bar{\bm{x}}=\bm{0}$, the McCormick inequalities~\eqref{eq:mc} reduce to $-1\leq\bar{w}_{ij}\leq1$, and hence $(\bar{\bm{x}},\bar{\bm{w}})\in\mathcal{F}_{\mathrm{RLT}}$.
    Since $H_{ij}^2=1$, the objective value at this point is $\frac{1}{m}\sum_{i\in I_m}\sum_{j\in J_m}H_{ij}^2=m$.

        We next prove $v_{\mathrm{SR}}(\bm{c})=\sqrt m$.
    Let $(\bm{x},\bm{w})\in\mathcal{F}_{\mathrm{SDP\text{-}RLT}}$, and let $\bm{W}$ satisfy \Cref{eq:w-link,eq:psd,eq:diag-rlt}.
    Let $\bm{Z}\in\mathbb{S}^n$ be given by
    \begin{equation}
        Z_{ii}=\frac{1}{2m}~(i\in I_m\cup J_m),
        \qquad
        Z_{ij}=Z_{ji}=-\frac{H_{ij}}{2m\sqrt m}~(i\in I_m,\ j\in J_m),
        \label{eq:hadamard-certificate}
    \end{equation}
    and all other entries equal to zero.
    Its principal submatrix indexed by $I_m\cup J_m$, with $I_m$ ordered before $J_m$, is
    \[
        \frac{1}{2m}
        \begin{pmatrix}
            \bm{I}&-\bm{H}/\sqrt m\\
            -\bm{H}^{\top}/\sqrt m&\bm{I}
        \end{pmatrix},
    \]
    which is positive semidefinite since $\bm{H}/\sqrt m$ is orthogonal.
    All entries of $\bm{Z}$ outside this submatrix are zero, and hence $\bm{Z}\succeq\bm{O}$. 
    Since $\bm{M}(\bm{W})\succeq\bm{O}$ by \Cref{eq:psd}, its principal submatrix $\bm{W}$ is also positive semidefinite, and hence $\bm{Z}\bullet\bm{W}\geq0$.
    By \Cref{eq:w-link,eq:diag-rlt,eq:hadamard-objective}, we have
    \[
        \bm{Z}\bullet\bm{W}
        =\frac{1}{2m}\sum_{i\in I_m\cup J_m}W_{ii}
        -\frac{1}{\sqrt m}\sum_{(i,j)\in E}c_{ij}w_{ij}
        \leq1-\frac{1}{\sqrt m}\sum_{(i,j)\in E}c_{ij}w_{ij},
    \]
    and hence $\sum_{(i,j)\in E}c_{ij}w_{ij}\leq\sqrt m$, which gives $v_{\mathrm{SR}}(\bm{c})\leq\sqrt m$.
    To see that this bound is attained, let $\bar{\bm{x}}=\bm{0}$ and
    \[
        \bar{w}_{ij}=H_{ij}/\sqrt m \quad (i\in I_m,\ j\in J_m),
        \qquad
        \bar{w}_{ij}=0 \quad (\text{otherwise}).
    \]
        Let $\bar{\bm{W}}\in\mathbb{S}^n$ have diagonal entries equal to one, $\bar{W}_{ij}=\bar{W}_{ji}=H_{ij}/\sqrt m$ for $i\in I_m$ and $j\in J_m$, and all other off-diagonal entries equal to zero.
    Then $\bar{W}_{ij}=0$ whenever $i\in I_m\cup J_m$ and $j\notin I_m\cup J_m$, so $\bar{\bm{W}}$ is block diagonal, with the block indexed by $I_m\cup J_m$ equal to
    \[
        \begin{pmatrix}
            \bm{I}&\bm{H}/\sqrt m\\
            \bm{H}^{\top}/\sqrt m&\bm{I}
        \end{pmatrix}
    \]
    when $I_m$ is ordered before $J_m$, and the other block is equal to the identity matrix.
    The former is positive semidefinite since $\bm{H}/\sqrt m$ is orthogonal, and hence $\bar{\bm{W}}\succeq\bm{O}$.

    Since $\bar{\bm x}=\bm 0$ and $|\bar w_{ij}|\le 1$ for every $(i,j)\in E$, the McCormick inequalities hold, and hence $(\bar{\bm x},\bar{\bm  w})\in\mathcal F_{\mathrm{RLT}}$.
    Thus \Cref{eq:w-link,eq:diag-rlt} hold and $\bm{M}(\bar{\bm{W}})\succeq\bm{O}$, so $(\bar{\bm{x}},\bar{\bm{w}})\in\mathcal{F}_{\mathrm{SDP\text{-}RLT}}$.
    Its objective value is $\frac{1}{m}\sum_{i\in I_m}\sum_{j\in J_m}H_{ij}^2/\sqrt m=\sqrt m$.
\end{proof}

\begin{remark}
    The equality $v_{\mathrm{SR}}(\bm{c})=\sqrt m$ established in the proof of Proposition \ref{prop:hadamard-gap} also follows from the bounds of \citet[Corollary~3.3, Lemma~3.5, and Section~6.1]{linial2007complexity}, and is therefore not new.
\end{remark}

\section{Conclusion}
\label{sec:conclusion}

We studied when the SDP-RLT relaxation is strictly stronger than the RLT relaxation for bilinear terms in box-constrained variables.
Membership in the SDP-RLT relaxation was reduced to a positive semidefinite completion problem whose specified entries correspond to the bilinear terms of the problem, and the feasible regions of the two relaxations were shown to coincide if and only if the graph formed by these terms is a forest.
For a forest, we obtain the completion from the classical result for chordal patterns; for a graph containing a cycle, we construct a point that is feasible for the RLT relaxation but whose partial matrix has no positive semidefinite completion. 
We also measured the resulting gap when the graph contains a complete bipartite subgraph.
For a bilinear objective given by a Hadamard matrix, the optimal values of the two relaxations differ by a factor that grows with the size of that subgraph.

\bibliography{sn-bibliography}

\end{document}